\documentclass{amsart}
\usepackage[utf8]{inputenc}
\usepackage[hidelinks]{hyperref}
\usepackage[inline]{enumitem}
\usepackage{amsthm, amsmath, amssymb, colonequals, comment, stmaryrd, tabularx, xcolor}

\newtheorem{corollary}{Corollary}
\newtheorem{lemma}{Lemma}
\newtheorem{proposition}{Proposition}
\newtheorem{theorem}{Theorem}
\newtheorem{problem}{Problem}

\theoremstyle{definition}

\newtheorem{example}{Example}

\theoremstyle{remark}
\newtheorem{remark}{Remark}

\DeclareMathOperator{\Aut}{Aut}
\DeclareMathOperator{\id}{id}

\begin{document}

\title{Non-associative versions of Hilbert's basis theorem}
\author{Per B\"ack}
\address[Per B\"ack]{Division of Mathematics and Physics, Box 883, SE-721  23  V\"aster\r{a}s, Sweden}
\email[corresponding author]{per.back@mdu.se}

\author{Johan Richter}
\address[Johan Richter]{Department of Mathematics and Natural Sciences, Blekinge Institute of Technology, SE-371 79 Karlskrona, Sweden}
\email{johan.richter@bth.se}


\begin{abstract}
We prove several new versions of Hilbert's basis theorem for non-associative Ore extensions, non-associative skew Laurent polynomial rings, non-associative skew power series rings, and non-associative skew Laurent series rings. For non-associative skew Laurent polynomial rings, we show that both a left and a right version of Hilbert's basis theorem hold. For non-associative Ore extensions, we show that a right version holds, but give a counterexample to a left version; a difference that does not appear in the associative setting.
\end{abstract}

\maketitle

\section{Introduction}
In commutative algebra and algebraic geometry, the classical Hilbert's basis theorem, which says that if $R$ is a unital, associative, commutative, and Noetherian ring, then $R[X]$ is also Noetherian, is of fundamental importance. This theorem can be generalized as follows: if $R$ is a unital, associative, left (right) Noetherian ring together with a ring automorphism $\sigma$ and $\sigma$-derivation $\delta$, then the Ore extension $R[X;\sigma,\delta]$, the Laurent polynomial ring $R[X^\pm;\sigma]$, the skew power series ring $R[[X;\sigma]]$, and the skew Laurent series ring $R((X;\sigma))$ are also left (right) Noetherian (many proofs can be found in the textbook \cite{GW04} by Goodearl and Warfield. Alternatively, see the proofs in this article). In \cite{BR22}, Hilbert's basis theorem for Ore extensions was further extended by the present authors to the case when $R$ is non-associative or hom-associative. In this article, we investigate versions of Hilbert's basis theorem for all the above rings in the non-associative setting where $\sigma$ is an additive surjection or bijection respecting 1, that is, $\sigma(1)=1$, and $\delta$ is an additive map where $\delta(1)=0$. Non-associative versions of Ore extensions were introduced by Nystedt, {\"O}inert, and Richter in \cite{NOR18}, and the non-associative skew Laurent polynomial rings used in this article were introduced by the present authors in \cite{BR24}.

We prove several new versions of Hilbert's basis theorem for non-associative skew Laurent polynomial rings, non-associative skew Laurent series rings, non-associative skew power series rings, and non-associative Ore extensions, thus generalizing results from \cite{BR22}. Unlike for associative rings, there is no simple equivalence between the left and right Hilbert's basis theorem; in fact for non-associative Ore extensions, only a right version holds in general. In more detail, for non-associative skew Laurent polynomial rings, we prove both a left and a right version of Hilbert's basis theorem (\autoref{thm:right-hilbert}). For non-associative Ore extensions, we prove a right version of Hilbert's basis theorem (\autoref{thm:ore-right-hilbert}) that generalizes the right version of Hilbert's basis theorem for non-associative Ore extensions proved in \cite{BR22} by applying to a non-associative Ore extension $R[X;\sigma, \delta]$, where we only require $\sigma$ to be an additive surjection with $\sigma(1)=1$ and $\delta$ to be an additive map with $\delta(1)=0$. In \autoref{ex:not-left-noetherian}, we show that the left version of Hilbert's basis theorem for non-associative Ore extensions in \cite{BR22} cannot be similarly generalized.   

Lastly, we show that under certain conditions, one can also prove a Hilbert's basis theorem for non-associative generalizations of skew power series rings (\autoref{thm:power-hilbert}) and skew Laurent series rings (\autoref{thm:skew-series-hilbert}).

The article is organized as follows:

In \autoref{sec:prel}, we provide conventions and preliminaries from non-associative ring theory (\autoref{subsec:non-assoc-ring-theory}). We also recall what skew Laurent polynomial rings and Ore extensions are (\autoref{subsec:skew-laurent-ore}), and how the definition of these rings can be extended to the non-associative setting.

In \autoref{sec:hilbert-basis-theorem}, we prove the above mentioned results on Hilbert's basis theorem: in \autoref{subsec:non-assoc-skew-laurent} for non-associative skew Laurent polynomial rings, in \autoref{subsec:hilbert-non-assoc-ore} for non-associative Ore extensions, and in \autoref{subsec:non-assoc-skew-power-laurent} for non-associative generalizations of both power series rings and skew Laurent series rings. 
\section{Preliminaries}\label{sec:prel}

\subsection{Non-associative ring theory}\label{subsec:non-assoc-ring-theory}
We denote by $\mathbb{N}$ the natural numbers, including zero. By a \emph{non-associative ring}, we mean a unital ring which is not necessarily associative. If $R$ is a non-associative ring, by a \emph{left $R$-module}, we mean an additive group $M$ equipped with a biadditive map $R\times M\to M$, $(r,m)\mapsto rm$ for any $r\in R$ and $m\in M$. A subset $B$ of $M$ is a basis if for any $m\in M$, there are unique $r_b\in R$ for $b\in B$, such that $r_b=0$ for all but finitely many $b\in B$, and $m=\sum_{b\in B}r_bb$. A left $R$-module that has a basis is called \emph{free}.

For a non-associative ring $R$, the \emph{associator} is the function $(\cdot,\cdot,\cdot)\colon R\times R\times R\to R$ defined by $(r,s,t)=(rs)t-r(st)$ for all $r,s,t\in R$. Using the associator we define three sets: $N_l(R)\colonequals \{r\in R\colon (r,s,t)=0 \text{ for all } s,t\in R\}$, $N_m(R)\colonequals \{s\in R\colon (r,s,t)=0 \text{ for all } r,t\in R\}$, and $N_r(R)\colonequals \{t\in R\colon (r,s,t)=0 \text{ for all } r,s\in R\}$. From the so-called associator identity 
\begin{equation*}
u(r,s,t)+(u,r,s)t+(u,rs,t)=(ur,s,t)+(u,r,st)
\end{equation*}
which holds for all $r,s,t,u\in R$, it follows that $N_l(R)$, $N_m(R)$, and $N_r(R)$ are all associative subrings of $R$. We also define $N(R)\colonequals N_l(R)\cap N_m(R)\cap N_r(R)$. 

If $R$ is a non-associative ring, recall that an element $u\in R$ is said to be left (right) invertible if there is $v\in R$ ($w\in R$) such that $vu=1$ ($uw=1$); in that case $v$ (or $w$) is called a left (or right) inverse of $u$. We let $R^\times$ denote the set of elements of $R$ that are both left and right invertible.

\begin{remark}\label{rem:unique-inverse}
Suppose $u\in N_m(R)\cap R^\times$. It is easy to show that $u$ has a unique left inverse $v$, $u$ has a unique right inverse $w$, and $v=w$. We let $u^{-1}$ denote the element $v=w$.
\end{remark}

By a non-associative ring $R$ being left (right) Noetherian, we mean that $R$ satisfies the ascending chain condition on left (right) ideals. In \cite{BR22}, we show that this is equivalent to all left (right) ideals being finitely generated and that the non-associative module theory parallels the associative case.

\subsection{Ore extensions and skew Laurent polynomial rings }\label{subsec:skew-laurent-ore}
Let us recall the definitions of (associative) Ore extensions and skew Laurent polynomial rings.

Let $R$ be an associative ring. We want to equip the ordinary polynomial ring $R[X]$ with a new multiplication satisfying $XR\subseteq RX+R$. This implies the existence of additive maps $\sigma,\delta\colon R\to R$ such that $Xr=\sigma(r)X+\delta(r)$ for any $r\in R$. The full multiplication is given by the biadditive extension of the relations
\begin{equation}
\left(rX^m\right)\left(sX^n\right)=\sum_{i\in\mathbb{N}}(r\pi_i^m(s))X^{i+n}\label{eq:ore-mult}
\end{equation}
for any $r,s\in R$ and $m,n\in\mathbb{N}$. Here $\pi_i^m$ denotes the sum of all $\binom{m}{i}$ possible compositions of $i$ copies of $\sigma$ and $m-i$ copies of $\delta$, where $\pi_i^m(s)\colonequals 0$ whenever $m<i$. For instance, $\pi_1^3=\sigma\circ\delta\circ\delta+\delta\circ\sigma\circ\delta+\delta\circ\delta\circ\sigma$. For the resulting structure to be an associative ring, it is necessary and sufficient that $\sigma$ is an endomorphism of $R$ and $\delta$ is a \emph{$\sigma$-derivation}, i.e. an additive map satisfying
\begin{equation*}
\delta(rs)=\sigma(r)\delta(s)+\delta(r)s 
\end{equation*}
for any $r,s\in R$. The resulting associative ring is denoted by $R[X;\sigma,\delta]$ and called an \emph{Ore extension} of $R$.

Similarly, we can equip the ordinary Laurent polynomial ring $R[X^\pm]$ with a new associative multiplication satisfying $XR=RX$ by using a ring automorphism $\sigma$. The multiplication is then defined by the biadditive extension of the relations
\begin{equation}
\left(rX^m\right)\left(sX^n\right)=(r\sigma^m(s))X^{m+n}\label{eq:laurent-mult}
\end{equation}	
for any $r,s\in R$ and $m,n\in\mathbb{Z}$. The resulting associative ring is denoted by $R[X^\pm;\sigma]$ and called a \emph{skew Laurent polynomial ring} over $R$. 

\subsection{Non-associative Ore extensions and non-associative skew Laurent polynomial rings}\label{subsec:non-assoc-ore}
In \cite{NOR18}, it was noted that the product \eqref{eq:ore-mult} gives a non-associative ring extension for any non-associative ring $R$ and any two additive maps $\sigma\colon R\to R$ and $\delta\colon R\to R$ satisfying $\sigma(1)=1$ and $\delta(1)=0$. 

To be precise, let $R$ be a non-associative ring. We denote by $R[X]$ the set of formal sums $\sum_{i\in \mathbb{N}}r_iX^i$ where $r_i\in R$ is zero for all but finitely many $i\in \mathbb{N}$, equipped with pointwise addition. Now, let $\sigma$ and $\delta$ be additive maps on $R$ satisfying $\sigma(1)=1$ and $\delta(1)=0$. The \emph{non-associative Ore extension} $R[X;\sigma, \delta]$ of $R$ is defined as the additive group $R[X]$ with multiplication defined by \eqref{eq:ore-mult}. One readily verifies that this makes $R[X;\sigma,\delta]$ a non-associative ring. It is associative if and only if $R$ is associative $\sigma$ is an endomorphism, and $\delta$ is a $\sigma$-derivation. It follows from results in \cite{NOR18} that $X^n\in N_m(R[X;\sigma,\delta])\cap N_r(R[X;\sigma,\delta])$ for any $n\in\mathbb{N}$.

\begin{example}\label{ex:weyl-algebra}
Let $T$ be a non-associative ring and let $R=T[Y]$. If $\delta\colon R\to R$ is a $T$-linear map where $\delta(1)=0$, then the non-associative Ore extension $R[X;\id_R, \delta]$ is called a \emph{non-associative Weyl algebra} in \cite{NOR18}.
\end{example}

Non-associative skew Laurent polynomial rings are defined in an analogous fashion to how non-associative Ore extensions are defined. 

Let $R$ be a non-associative ring. We denote by $R[X^\pm]$ the set of formal sums $\sum_{i\in \mathbb{Z}}r_iX^i$ where $r_i\in R$ is zero for all but finitely many $i\in \mathbb{Z}$, equipped with pointwise addition. Now, let $\sigma$ be an additive bijection on $R$ respecting $1$. The \emph{non-associative skew Laurent polynomial ring} $R[X;\sigma, \delta]$ over $R$ is defined as the additive group $R[X^\pm]$ with multiplication defined by \eqref{eq:laurent-mult}. One readily verifies that this makes $R[X^\pm;\sigma]$ a non-associative ring. It is associative if and only if $R$ is associative and $\sigma$ is a ring automorphism. By Proposition 13 in \cite{BR24}, $X^n\in N_m(R[X^\pm;\sigma)\cap N_r(R[X^\pm;\sigma])$ for any $n\in\mathbb{Z}$.

\begin{comment}
\begin{remark}
Both non-associative Ore extensions and non-associative skew Laurent polynomial rings are examples of \emph{Ore monoid rings} as defined by Nystedt, {\"O}inert, and Richter \cite{NOR19}.    
\end{remark}
\end{comment}

\begin{example}\label{ex:complex-skew-laurent}
	On the ring $\mathbb{C}$ we can define $\sigma_q(a+bi)= a+qbi$ for any $a,b\in \mathbb{R}$ and $q\in\mathbb{R}^\times$. Then $\sigma$ is an additive bijection that respects $1$, and  we can accordingly define $\mathbb{C}[X^{\pm}; \sigma_q]$. Moreover, $\sigma_q$ is a ring automorphism if and only if $q=\pm 1$, and so $\mathbb{C}[X^{\pm};\sigma_q]$ is associative if and only if $q=\pm 1$.
\end{example}

\begin{example}\label{ex:quantum-torus}   
Let $T$ be a non-associative ring, $q\in Z(T)^\times$, and $R=T[Y]$. Since $Z(T)$ is associative, $q$ has a unique two-sided multiplicative inverse. Define a ring automorphism $\sigma_q\colon R\to R$ by the $T$-algebra extension of the relation $\sigma_q(Y)=qY$. The \emph{non-associative quantum torus} over $T$ is the skew Laurent polynomial ring $R[X^\pm;\sigma_q]$. It is associative if and only if $T$ is associative.
\end{example}

\begin{example}\label{ex:cayley-staralgebra}
Let $A$ be any of the real, non-associative Cayley--Dickson algebras $\mathbb{R}, \mathbb{C}, \mathbb{H}, \ldots$ with the ring anti-automorphism $*$ given by the conjugation map. Then $*$ is a ring automorphism on $A$ if and only if $A=\mathbb{R}$ or $\mathbb{C}$ if and only if $A[X^\pm;*]$ is associative.
\end{example}

\section{Hilbert's basis theorem}\label{sec:hilbert-basis-theorem}
In this section, we extend Hilbert's basis theorem to non-associative settings.

\subsection{Hilbert's basis theorem for non-associative skew Laurent polynomial rings}\label{subsec:non-assoc-skew-laurent}
\begin{theorem}\label{thm:right-hilbert}
Let $R$ be a non-associative ring with an additive bijection $\sigma$ that respects $1$. If $R$ is left (right) Noetherian, then so is $R[X^\pm;\sigma]$.
\end{theorem}

\begin{proof}
The following proof is a minor adaptation of the proof of Proposition 2.5 in \cite{Bel87}. 

Let $R$ be left Noetherian. For a left ideal $I$ of $S\colonequals R[X^\pm;\sigma]$ and a positive integer $n$, define
\begin{equation*}
l_n(I)=\left\{r_0\in R\colon \sum_{i={-n+1}}^{0} r_iX^i\in I \text{ for some } r_i\in R \text{ where } -n+1\leq i < 0\right\}.	
\end{equation*}
It is clear that each $l_n(I)$ is a left ideal of $R$ and that $l_1(I)\subseteq l_2(I)\subseteq\cdots$.

Suppose that $I$ and $J$ are left ideals of $S$ with $J\subseteq I$ and $l_n(I)=l_n(J)$ for each positive integer $n$: we claim that $I=J$. If not, then there must be an element $s=\sum_{i=a}^bs_iX^i$ (with $s_i\in R$) in $I\backslash J$ with $b-a$ as small as possible. Since $X^b(X^{-b}s)=s$, it is not true that $X^{-b}s$ belongs to $J$, but $X^{-b}s$ belongs to $I$, so we may assume $b=0$. Thus $s=\sum_{i=a}^{0}s_iX^i$, and so $s_0\in l_{1-a}(I)=l_{1-a}(J)$. This means there is an element $t=s_0+\sum_{i=a}^{-1}\tilde{s}_i\in J\subseteq I$. Hence $X(s-t)=\sum_{i=a+1}^0u_iX^i$ for some $u_i\in R$. By the assumption on $s$, we get that $X(s-t)\in J$. Hence $Xs\in J$, and so $X^{-1}(Xs)=s\in J$, which is a contradiction. Thus $I=J$.

Now suppose that $I_1\subseteq I_2\subseteq\cdots$ is a chain of left ideals of $S$. Clearly $l_1(I_1)\subseteq l_2(I_2)\subseteq\cdots$, so by the left Noetherianness of $R$, there is a $k$ such that $l_k(I_k)=l_n(I_n)$ for all $n\geq k$. It is clear that in fact $l_k(I_k)= l_n(I_m)$ for all $n, m\geq k$. Consider the chains $l_j(I_1)\subseteq l_j(I_2)\subseteq\cdots$, for $1\leq j\leq k-1$. Again by Noetherianness, there is an $n$, which we may choose bigger than $k$, such that $l_j(I_n)=l_j(I_m)$ for $m\geq n$ and all $j$ with $1\leq j\leq k-1$. But this equality already holds for $j\geq k$, so in fact $l_j(I_n)=l_j(I_m)$ for all $m\geq n$ and all $j$, and so by the previous paragraph, $I_n=I_m$ for all $m\geq n$. This shows that $S$ is left Noetherian. 

The right case is similar.
\end{proof}

\begin{comment}
\begin{proof}
Let $R$ be a non-associative, right Noetherian ring with an additive bijection $\sigma$ that respects $1$. By \autoref{thm:ore-right-hilbert}, we know that $T\colonequals R[X; \sigma]$ is right Noetherian. 
	
	Let $I$ be a right ideal of $S\colonequals R[X^\pm;\sigma]$. Then $I \cap T$ is a right ideal of $T$. We claim that $I= (I\cap T)S$. If $p\in I$ then $p=p_m X^m + \cdots + p_n X^n$ where $p_i\in R$ and $i,m,n$ are integers with $m \leq n$. Then $p=(pX^{-m})X^m$ and $pX^{-m} \in I \cap T$, so $p \in (I \cap T)S $. If $p\in (I\cap T)S$, then obviously $p\in I$.
	
	Now suppose we have an ascending chain of right ideals, $I_1 \subseteq I_2 \ldots$, in $S$. Set $J_i = I_i \cap T$. Then $I_i = J_iS$ for all $i$ and the $J_i$s form an ascending chain of right ideals of $T$. Since $T$ is right Noetherian there is some $j\in\mathbb{N}_{>0}$ such that $J_j=J_{j+1}= J_{j+2}=\ldots$. Then also $I_j=I_{j+1}=I_{j+2}=\ldots$, so $S$ is right Noetherian. 
\end{proof}
\end{comment}

\begin{example}
If $R$ is an associative, commutative, Noetherian ring, then the matrix ring $M_n(R)$ is Noetherian for any $n\in\mathbb{N}_{>0}$ (see e.g. Proposition 1.6 in \cite{GW04}). Hence $M_n(R)[X^\pm;\sigma]$ where $\sigma$ is the ring anti-automorphism given by the transpose operation, is Noetherian.	
\end{example}

\begin{remark}
Any non-associative division ring is Noetherian, and so by \autoref{thm:right-hilbert}, the non-associative skew Laurent polynomial rings in \autoref{ex:complex-skew-laurent} and in \autoref{ex:cayley-staralgebra} whenever $A=\mathbb{R}, \mathbb{C}, \mathbb{H}$ or $\mathbb{O}$, are Noetherian. 
\end{remark}

\begin{remark}
By \autoref{thm:right-hilbert}, the non-associative quantum torus $T[Y^\pm][X^\pm;\sigma_q]$ in \autoref{ex:quantum-torus} is left (right) Noetherian if $T$ is left (right) Noetherian.
\end{remark}

\begin{comment}
 Given an associative algebra $A$ over a field $K$ of characteristic different from two, we may define a unital, non-associative $K$-algebra $A^+$ by using the \emph{Jordan product} $\{\cdot,\cdot\}\colon A^+\to A^+$, given by $\{a,b\}\colonequals \frac{1}{2}\left(ab + ba\right)$ for any $a,b\in A$. $A^+$ is then a \emph{Jordan algebra}, i.e. a commutative algebra where any two elements $a$ and $b$ satisfy the \emph{Jordan identity}, $\left\{\{a,b\}\{a,a\}\right\}=\left\{a,\{b,\{a,a\}\}\right\}$. Since inverses on $A$ extend to inverses on $A^+$, we see that if $A=\mathbb{H}$, then $A^+$ is also Noetherian. Using the standard notation $i,j,k$ for the quaternion units with defining relation $i^2=j^2=k^2=ijk=-1$, we see that $\mathbb{H}^+$ is not associative as e.g. $(i,i,j)_{\mathbb{H}^+}=\{\{i,i\},j\}-\{i,\{i,j\}\}=-j$. 

\begin{example}
Let $\sigma\in\Aut_\mathbb{R}(\mathbb{H})$. Then $\sigma\in\Aut_\mathbb{R}(\mathbb{H}^+)$, and so by \autoref{thm:right-hilbert}, $\mathbb{H}^+[X^\pm;\sigma]$ is Noetherian. 
\end{example}
\end{comment}

By \autoref{thm:right-hilbert}, it is immediate that if $R$ is left (right) Noetherian, then so are all iterated non-associative skew Laurent polynomial rings of $R$ where all the $\sigma$s are additive bijections respecting $1$. Now, if $\sigma_1,\ldots,\sigma_n$ are pairwise commuting additive bijections of $R$ respecting $1$, then we may construct an iterated non-associative skew Laurent polynomial ring of $R$ as follows (see e.g. Exercise 1W in \cite{GW04} for the associative case, which is nearly identical). First, we set $S_1\colonequals R[X_1^\pm;\sigma_1]$. Then $\sigma_2$ extends to an additive bijection $\widehat{\sigma}_2$ on $S_1$ respecting $1$, defined by $\widehat{\sigma}_2(rX_1^m)=\sigma_2(r)X_1^m$ for any $m\in\mathbb{Z}$. Next, we set $S_2\colonequals S_1[X_2^\pm;\widehat\sigma_2]$. Once $S_i$ has been constructed for some $i<n$, we construct $S_{i+1}\colonequals S_i[X_{i+1}^\pm;\widehat{\sigma}_{i+1}]$ where $\widehat{\sigma}_{i+1}$ is the additive bijection on $S_i$ defined by $\widehat{\sigma}_{i+1}(rX_1^{m_1}\cdots X_n^{m_n})=\sigma_{i+1}(r)X_1^{m_1}\cdots X_n^{m_n}$. We denote by $R[X_1^\pm,\ldots,X_n^\pm;\sigma_1,\ldots,\sigma_n]$ the resulting iterated non-associative skew Laurent polynomial ring $R[X_1^\pm;\sigma_1]\cdots[X_n^\pm;\widehat{\sigma}_n]$. This is a generalization of the non-associative skew Laurent polynomial rings defined in \cite{NO20}; their construction corresponds exactly to the case when $\sigma_1,\ldots,\sigma_n$ are ring automorphisms.

\begin{corollary}\label{cor:commuting-multi-hilbert}
Let $R$ be a non-associative ring with pairwise commuting additive bijections $\sigma_1,\ldots,\sigma_n$ respecting $1$. If $R$ is left (right) Noetherian, then so is $R[X_1^\pm,\ldots,X_n^\pm;\sigma_1,\ldots,\sigma_n]$.
\end{corollary}

\subsection{Hilbert's basis theorem for non-associative Ore extensions}\label{subsec:hilbert-non-assoc-ore}
We will now see to what extent Hilbert's basis theorem can be extended to non-associative Ore extensions (see \autoref{subsec:non-assoc-ore}).  

\begin{lemma}\label{lem:ore-module}
Let $R$ be a non-associative ring with an additive surjection $\sigma$ that respects $1$ and an additive map $\delta$ where $\delta(1)=0$. Then, for any $n\in\mathbb{N}$, $\sum_{i=0}^n X^iR=\sum_{i=0}^n RX^i$ as right $R$-modules in $R[X;\sigma,\delta]$.
\end{lemma}

\begin{proof}
We show that as sets, $\sum_{i=0}^nX^iR=\sum_{i=0}^nRX^i$. Since $\sum_{i=0}^nRX^i$ is clearly closed under addition and multiplication from the right by an element from $R$, it is indeed a right $R$-module. We also see that $\sum_{i=0}^nX^iR\subseteq\sum_{i=0}^nRX^i$, so we only need to show that the other inclusion holds as well. We prove this by induction on $n$.

Base case ($n=0$): We must show that $RX^0\subseteq X^0R$. However, since $X^0=1$, this is immediate.

Induction step ($n+1$): Assume that $\sum_{i=0}^nRX^i\subseteq \sum_{i=0}^n X^iR$ and let $p\in \sum_{i=0}^{n+1}RX^i$. By definition, $p=rX^{n+1}+$ [lower order terms] for some $r\in R$. Since $\sigma$ is surjective, so is $\sigma^{n+1}$, and so there exists $r'\in R$ such that $\sigma^{n+1}(r')=r$. Then $p-X^{n+1}r'\in \sum_{i=0}^nRX^i\subseteq \sum_{i=0}^nX^iR$, so we must have $p\in\sum_{i=0}^{n+1}X^iR$. Hence $\sum_{i=0}^{n+1}RX^i\subseteq \sum_{i=0}^{n+1}X^iR$, and the induction step is done. 
\end{proof}

\begin{theorem}\label{thm:ore-right-hilbert} 
Let $R$ be a non-associative ring with an additive surjection $\sigma$ that respects $1$ and an additive map $\delta$ where $\delta(1)=0$. If $R$ is right Noetherian, then so is $R[X;\sigma, \delta]$.
\end{theorem}

\begin{proof}
This proof is an adaptation of a proof in \cite{GW04} to our setting. We wish to show that any right ideal of $R[X;\sigma, \delta]$ is finitely generated. Since the zero ideal is finitely generated, it is sufficient to show that any non-zero right ideal $I$ of $R[X;\sigma, \delta]$ is finitely generated. Let $J$ consist of the zero element and all leading coefficients of polynomials in $I$, i.e. $J\colonequals \{r\in R\colon rX^d+ r_{d-1}X^{d-1}+\dots +r_0\in I, r_{d-1},\dots,r_0\in R\}$. We claim that $J$ is a right ideal of $R$. First, one readily verifies that $J$ is an additive subgroup of $R$. Now, let $r\in J$ and $s\in R$ be arbitrary. Then there is some polynomial $p=rX^d+[\text{lower order terms}]$ in $I$. Moreover, there exists $s'\in R$ such that $\sigma^d(s')=s$. Hence $ps'=(rX^d)s'+[\text{lower order terms}]=\left(r\sigma^d(s')\right)X^d+[\text{lower order terms}]=(rs)X^d+[\text{lower order terms}]$, which is an element of $I$ since $p$ is. Therefore, $rs\in J$, so $J$ is a right ideal of $R$. 
	
Since $R$ is right Noetherian and $J$ is a right ideal of $R$, $J$ is finitely generated, say by $\{r_1,\dots,r_k\}\subseteq J$. All the elements $r_1,\dots,r_k$ are assumed to be non-zero, and moreover, each of them is a leading coefficient of some polynomial $p_i\in I$ of degree $m_i$. Put $m=\max(m_1,\dots,m_k)$. Then each $r_i$ is the leading coefficient of $p_i X^{m-m_i}=r_i X^{m_i}\cdot X^{m-m_i}+[\text{lower order terms}]=r_iX^m+[\text{lower order terms}]$, which is an element of $I$ of degree $m$.
	
Let $M\colonequals \sum_{i=0}^{m-1} RX^i$. By \autoref{lem:ore-module}, $\sum_{i=0}^{m-1} RX^i = \sum_{i=0}^{m-1} X^iR$ as right $R$-modules. Hence $M$ is finitely generated, and any finitely generated right $R$-module is Noetherian. Now, since $I$ is a right ideal of the ring $R[X;\sigma,\delta]$ which contains $R$, in particular, it is also a right $R$-module. Hence $I\cap M$ is a submodule of $M$, and since $M$ is a Noetherian right $R$-module, $I\cap M$ is finitely generated, say by the set $\{q_1,\dots, q_t\}$.
	
Let $I_0$ be the right ideal of $R[X;\sigma,\delta]$ generated by $\left\{p_1X^{m-m_1},\dots,p_kX^{m-m_k},\right.$
\\$\left. q_1,\ldots,q_t\right\}$. Since all the elements in this set belong to $I$, we have that $I_0\subseteq I$. We claim that $I\subseteq I_0$. In order to prove this, pick any element $p'\in I$.
	
Base case ($\mathsf{P}(m)$): If $\deg p'<m$, $p'\in M=\sum_{i=0}^{m-1} RX^i$, so $p'\in I\cap M$. On the other hand, the generating set of $I\cap M$ is a subset of the generating set of $I_0$, so $I\cap M\subseteq I_0$, and therefore $p'\in I_0$. 
	
Induction step ($\forall n\geq m\ (\mathsf{P}(n)\rightarrow \mathsf{P}(n+1))$): Assume $\deg p'= n\geq m$ and that $I_0$ contains all elements of $I$ with $\deg < n$. We want to show that $I_0$ contains $p'$ as well. Let $r'$ be the leading coefficient of $p'$, so that we have $p'=r'X^n+[\text{lower order terms}]$. Since $p'\in I$ by assumption, $r'\in J$. We then claim that $r'=\sum_{i=1}^k\sum_{j=1}^{k'}(\cdots((r_is_{ij1})s_{ij2})\cdots)s_{ijj}$ for some $k'\in\mathbb{N}_{>0}$ and some $s_{ij1},\ldots, s_{ijj}\in R$. First, we note that since $J$ is generated by $\{r_1,\dots,r_k\}$, it is necessary that $J$ contains all elements of that form. Secondly, we see that subtracting any two such elements or multiplying any such element from the right with one from $R$ again yields such an element, and hence the set of all elements of this form is not only a right ideal containing $\{r_1,\dots,r_k\}$, but also the smallest such to do so. 
	
Recall that $p_iX^{m-m_i}=r_iX^m+[\text{lower order terms}]$. There exists $s'_{ij\ell}$ such that $\sigma^m(s'_{ij\ell})=s_{ij\ell}$, so $\left(p_i X^{m-m_i}\right)s'_{ij\ell}=(r_is_{ij\ell})X^m+[\text{lower order terms}]$. Set $c_{ij}\colonequals \left(\cdots\left(\left(\left(p_iX^{m-m_i}\right)s'_{ij1}\right)s'_{ij2}\right)\cdots\right)s'_{ijj}$. Since $p_iX^{m-m_i}$ is a generator of $I_0$, $c_{ij}$ is an element of $I_0$ as well, and therefore so is the element $q\colonequals \sum_{i=1}^k\sum_{j=1}^{k'} c_{ij}X^{n-m}=r'X^n + [\text{lower order terms}].$ However, as $I_0\subseteq I$, we also have $q\in I$, and since $p'\in I$, $(p'-q)\in I$. Now, $p'=r'X^n+[\text{lower order terms}]$, so $\deg(p'-q)<n$, and therefore $(p'-q)\in I_0$. This shows that $p'=(p'-q)+q$ is an element of $I_0$ as well, and thus $I=I_0$, which is finitely generated.
\end{proof}

\begin{remark}
By \autoref{thm:ore-right-hilbert}, the non-associative Weyl algebra $T[Y][X;\id_R,\delta]$ in \autoref{ex:quantum-torus} is right Noetherian if $T$ is right Noetherian. By Theorem 1 in \cite{BR22}, $T[Y][X;\id_R,\delta]$ is left Noetherian if $T$ is left Noetherian.
\end{remark}

\begin{example}
Let $f\colon \mathbb{N}\to \mathbb{N}$ be any surjection with $f(0)=0$. Suppose $K$ is a field and set $R=K[Y]$. Let $\sigma$ be the additive surjection on $R$ defined by $\sigma(rY^n)=rY^{f(n)}$ for any $r\in R$ and $n\in\mathbb{N}$, and let $\delta$ be the ordinary derivative on $R$. Then $\sigma$ respects $1$ and $\delta(1)=0$. By \autoref{thm:ore-right-hilbert}, $R[X;\sigma, \delta]$ is right Noetherian.
\end{example}

\begin{comment}
\begin{example}
Let $R=\mathbb{R}[Y]$ and let $\sigma$ be the map on $R$ defined such that $\sigma(p(Y))$  for any polynomial $p(Y)\in R$ has twice the coefficient of $Y$ that $p(Y)$ has, but all other coefficients are the same (e.g. $\sigma( 2-Y+3Y^2) = 2-2Y+3Y^2)$, and let $\delta$ be the ordinary derivative. Then $\sigma$ is an additive bijection that respects $1$, and $\delta$ is an additive map where $\delta(1)=0$. From \autoref{thm:ore-right-hilbert}, $R[X;\sigma, \delta]$ is right Noetherian. 
\end{example}
\end{comment}
We can relate the ideals of a non-associative skew Laurent polynomial ring to a subring that is a non-associative Ore extension. 

\begin{proposition}\label{prop:skew-ideals}
Let $R$ be a non-associative ring with an additive bijection $\sigma$ that respects $1$. Set $S\colonequals R[X^\pm;\sigma]$ and $T\colonequals R[X;\sigma,0]$. If $I$ is a left ideal of $S$, then $I=S(I\cap T)$. If $I$ is a right ideal of $S$, then $I=(I\cap T)S$. 	
\end{proposition}

\begin{proof}
Let $I$ be a right ideal of $S$. Then $I \cap T$ is a right ideal of $T$. We claim that $I= (I\cap T)S$. If $p\in I$ then $p=p_m X^m + \cdots + p_n X^n$ where $p_i\in R$ and $i,m,n$ are integers with $m \leq n$. Then $p=(pX^{-m})X^m$ and $pX^{-m} \in I \cap T$, so $p \in (I \cap T)S $. If $p\in (I\cap T)S$, then obviously $p\in I$.	

The left case is similar.
\end{proof}

\autoref{prop:skew-ideals} shows that if $R[X;\sigma,0]$ is left (right) Noetherian, then so is $R[X^\pm;\sigma]$. Combined with \autoref{thm:ore-right-hilbert}, this can be used to prove the right version of \autoref{thm:right-hilbert}. In the associative case, one can prove the left version of \autoref{thm:right-hilbert} similarly by using a left version of \autoref{thm:ore-right-hilbert}. However, in the following example, we show that a left version of \autoref{thm:ore-right-hilbert} is not true. 

\begin{example}\label{ex:not-left-noetherian}
Let $R=K[Y,Z]$ where $K$ is a field. Set $U\colonequals \{1,3,5,\ldots\}$ and let $V$ be the set $U\times\mathbb{N}$. Then there exist bijections $f\colon U\to\{1,2,3,\ldots\}$ and $g=(g_1,g_2)\colon \{2,4,6,\ldots\}\to V$. Define a map $\sigma$ on the monomials of $R$ as follows: $\sigma(1)=1$, $\sigma(Y^iZ^j)=Y^{2i}Z^j$ if $i>0$, $\sigma(Z^j)=Z^{f(j)}$ if $j$ is odd, and $\sigma(Z^j)=Y^{g_1(j)}Z^{g_2(j)}$ if $j$ is even. Extend $\sigma$ $K$-linearly to all polynomials in $R$. Then $\sigma$ is an additive bijection that respects $1$. Note that the ideal, $J$, of $R$ generated by $Y$ is mapped to the ideal generated by $Y^2$ by $\sigma$. Set $T\colonequals R[X;\sigma,0]$ and let $I=\{\sum_{i\in\mathbb{N}}r_iX^i\in T\colon r_i\in J\text{ for all } i\}$. Then $I$ is a left ideal of $T$. We claim that $I$ is not finitely generated. 

For suppose that $I$ is generated as a left ideal by $p_1,\ldots,p_n$ for some $n$. Let $m$ be the maximal degree in $X$ of $p_1,\ldots,p_n$. Then $YX^{m+1}$ is in the left ideal generated by these generators. Hence there are $s_i,t_{i,1}, t_{i,2},\ldots\in T$ such that $YX^{m+1}=\sum_{i=1}^ns_ip_i+\sum_{i=1}^nt_{i,1}(t_{i,2}p_i)+\cdots$. There must exist terms on the right of degree at least $m+1$. Note that if a term on the right has degree $m+1$, then its coefficients belong to the ideal generated by $Y^2$. This would mean that the coefficient on the left of degree $m+1$ also belongs to the ideal of $R$ generated by $Y^2$. This is a contradiction, so there cannot exist such a finite set of generators. 
\end{example}

\subsection{Hilbert's basis theorem for non-associative skew power series rings and non-associative Laurent series rings}\label{subsec:non-assoc-skew-power-laurent}
Let $R$ be a non-associative ring with an additive bijection $\sigma$ that respects $1$. We can define a \emph{non-associative skew power series ring} $R[[X;\sigma]]$ by simply equipping the set of formal power series $\sum_{i=0}^{\infty} r_i X^i$, where $r_i\in R$, with the usual pointwise addition and the multiplication defined by $\left(rX^m\right)\left(sX^n\right) = \left(r\sigma^{m}(s)\right)X^{m+n}$ for any $r,s\in R$ and $m,n\in\mathbb{N}$ (extended in the obvious way). In particular, this makes $R[[X;\sigma]]$ a unital, non-associative ring. 

We define the \emph{order} of a non-zero element of $R[[X;\sigma]]$ to be the least power of $X$ with a non-zero coefficient and that coefficient to be the \emph{leading coefficient}. Like usual in formal power series rings, we can define the value of an infinite series $\sum_{i=0}^{\infty} p_i$ as long as the order of the $p_i$s goes to infinity. 

\begin{theorem}\label{thm:power-hilbert} 
	Let $R$ be an associative ring with an additive surjection $\sigma$ that respects $1$. If $R$ is right Noetherian, then so is $R[[X;\sigma]]$.
\end{theorem}

\begin{proof}
Let $R$ be a right Noetherian ring satisfying the conditions of the theorem. Denote the leading coefficient of an element $p\in S\colonequals R[[X;\sigma]]$ by $c(p)$. 
	
Let $I$ be an arbitrary non-zero right ideal of $S$. Let $J$ be a set consisting of $0$ and the leading coefficients of non-zero elements in $I$. It is not difficult to see that $J$ is a right ideal of $R$ and thus it is finitely generated. 

Let $p_1$ be a non-zero element of $I$, such that no non-zero element of $I$ has lower order. Define inductively $p_2, \ldots, p_n$ such that the $p_{i+1}$ has minimal order among all elements in $I$  such that $c(p_{i+1})$ does not lie in the right ideal generated by $c(p_1), \ldots , c(p_i)$. This process must stop after finitely many steps. We claim that $p_1, \ldots, p_n$ generate $I$. Let an element $q\in I$ be given. Clearly there is some combination $\sum_{i=1}^{n} p_ik_{1,i}$ where $k_{1,i}\in S$ and either $k_{1,i}=0$ or the order of $k_{1,i}$ equals the order of $q$, that has the same leading coefficient and the same order as $q$. Then $q'=q-\sum_{i=1}^{n} p_ik_{1,i}$ is an element of $I$ of higher order than $q$. We can then find $k_{2,1},\ldots, k_{2,n}\in S$ and either $k_{2,i}=0$ or the order of $k_{2,i}$ equals the order of $q'$, such that $q'-\sum_{i=1}^{n} p_ik_{2,i}$ is an element of $I$ of yet higher order. Continuing this process we can write $q = p_1\sum_{\ell=1}^{\infty} k_{\ell, 1}+\cdots+p_n\sum_{\ell=1}^{\infty} k_{\ell, n}$, showing that $q$ belongs to the right ideal generated by $p_1, \ldots, p_n$. 
\end{proof}

Similarly, one can define a non-associative skew Laurent series ring $R((X;\sigma))$.

\begin{theorem}\label{thm:skew-series-hilbert}
Let $R$ be an associative ring with an additive bijection $\sigma$ that respects $1$. If $R$ is right Noetherian, then so is $R((X;\sigma))$.
\end{theorem}

\begin{proof}
Let $R$ be a right Noetherian ring satisfying the conditions of the theorem. Denote the leading coefficient of an element $p\in S\colonequals R((X;\sigma))$ by $c(p)$. 

Let $I$ be an arbitrary non-zero right ideal of $S$. Let $J$ be a set consisting of $0$ and the leading coefficients of non-zero elements in $I$. Then $J$ is a right ideal of $R$ and thus finitely generated. Let $p_1, \ldots, p_n$ be elements in $I$ such that $c(p_1), \ldots, c(p_n)$ generate $J$. By the exact same argument as in the proof of \autoref{thm:power-hilbert}, $p_1, \ldots, p_n$ generate $I$. 
\end{proof}

\begin{problem}
Can one generalize the above two theorems for $R$ non-associative?	
\end{problem}

We note that in a non-associative ring, the right ideal generated by elements $a_1, \ldots , a_n$ is not equal to the set of elements of the form $a_1r_1+ \cdots +a_n r_n$, which causes our proof strategy to fail in that case. 

\begin{problem}
Can one prove a left version of the above two theorems?	
\end{problem}

\section*{Acknowledgements}
We would like to thank the anonymous referee for valuable comments on the manuscript as well as Patrik Lundstr{\"o}m and Stefan Wagner for discussions.


\begin{thebibliography}{99}
\bibitem{BR22}
B{\"a}ck,~P. and Richter,~J.: 
\emph{Hilbert's basis theorem for non-associative and hom-associative Ore extensions}, 
Algebr. Represent. Theory \textbf{26}, 1051--1065 (2023). 

\bibitem{BR24}
B{\"a}ck,~P. and Richter,~J.: 
\emph{Simplicity of non-associative skew Laurent polynomial rings},
\href{https://arxiv.org/abs/2207.07994}{\tt arXiv:2207.07994}.

\bibitem{Bel87}
Bell,~A.~D.:
\emph{Localization and ideal theory in Noetherian strongly group-graded rings},
J. Algebra \textbf{105}, 76--115 (1987).

\bibitem{GW04}
Goodearl,~K.~R. and Warfield,~R.~B.:
\emph{An introduction to noncommutative Noetherian rings},
2nd ed., Cambridge University Press, Cambridge U.K. (2004).

\begin{comment}
\bibitem{Hil90}
Hilbert,~D.:
\emph{Ueber die Theorie der algebraischen Formen},
Math. Annalen \textbf{36}, 473--534 (1890).
\end{comment}

\bibitem{NO20}
Nystedt~, P. and {\"O}inert,~J.:
\emph{Simple graded rings, nonassociative crossed products and Cayley-Dickson doublings},
J. Algebra Appl. \textbf{19}(12) (2020).
	
\bibitem{NOR18}
Nystedt,~P., {\"O}inert,~J., and Richter,~J.:
\emph{Non-associative Ore extensions},
Isr. J. Math. \textbf{224}(1), 263--292 (2018). 
\begin{comment}
\bibitem{NOR19} 
Nystedt,~P., {\"O}inert,~J., and Richter,~J.: 
\emph{Simplicity of Ore monoid rings}, 
J. Algebra \textbf{530}, 69--85 (2019).
\end{comment}
\end{thebibliography}
\end{document}